\documentclass[12pt]{amsart}

\usepackage{amssymb}
\usepackage{amsfonts}
\usepackage{graphicx}
\usepackage{fullpage}
\usepackage{MnSymbol}
\usepackage{caption}
\usepackage{subcaption}

\DeclareCaptionLabelFormat{r-parens}{\textbf{#2)}} 

\usepackage{color}
\usepackage{tikz}
\usetikzlibrary{calc}
\usetikzlibrary{arrows}
\usetikzlibrary{decorations.markings}

\def\gjoin{\makebox[1.1\width][r]{\raisebox{1pt}{$\diamondplus$}}}
\def\mybox{\makebox[1.1\width][r]{\raisebox{1pt}{$\Box$}}}

\newtheorem{theorem}{Theorem}[section]
\newtheorem{lemma}[theorem]{Lemma}
\theoremstyle{definition}

\theoremstyle{definition}

\theoremstyle{remark}

\numberwithin{equation}{section}
\theoremstyle{plain}

\newtheorem{question}[theorem]{Question}
\newtheorem{cor}[theorem]{Corollary}

\newtheorem{prop}[theorem]{Proposition}

\begin{document}

\tikzset{
    third arrow/.style={
    decoration={markings,mark=at position 0.67 with {\arrow[scale = 2]{>}}},
    postaction={decorate},
    shorten >=0.4pt}}

\title{Fool's Solitaire on Joins and Cartesian Products of Graphs}
\author{Sarah Loeb}
\address[Sarah Loeb]{Department of Mathematics, University of Illinois at Urbana--Champaign,  Urbana, IL, 61801}
\email{sloeb2@illinois.edu}

\author{Jennifer Wise}
\address[Jennifer Wise]{Department of Mathematics, University of Illinois at Urbana--Champaign,  Urbana, IL, 61801}
\email{jiwise2@illinois.edu}

\begin{abstract}
\emph{Peg solitaire} is a game generalized to connected graphs by Beeler and Hoilman. In the game pegs are placed on all but one vertex. If $xyz$ form a 3-vertex path and $x$ and $y$ each have a peg but $z$ does not, then we can remove the pegs at $x$ and $y$ and place a peg at $z$. By analogy with the moves in the original game, this is called a \emph{jump}. The goal of the peg solitaire game on graphs is to find jumps that reduce the number of pegs on the graph to 1. 

Beeler and Rodriguez proposed a variant where we instead want to maximize the number of pegs remaining when no more jumps can be made. Maximizing over all initial locations of a single hole, the maximum number of pegs left on a graph $G$ when no jumps remain is the fool's solitaire number $F(G)$.
We determine the fool's solitaire number for the join of any graphs $G$ and $H$. For the cartesian product, we determine $F(G \mybox K_k)$ when $k \ge 3$ and $G$ is connected and show why our argument fails when $k=2$. Finally, we give conditions on graphs $G$ and $H$ that imply $F(G \mybox H) \ge F(G) F(H)$.
\end{abstract}

\maketitle


\section{Introduction}
\emph{Peg solitaire} is a game generalized to connected graphs by Beeler and Hoilman~\cite{BH}. In the peg solitaire game on graphs, each vertex except one starts with a peg. Vertices without pegs are said to be \emph{holes}. If adjacent vertices $x$ and $y$ have pegs, and $z$ adjacent to $y$ is a hole, then we may \emph{jump} the peg at $x$ over the peg at $y$ and into the hole at $z$. This removes the peg at $y$ so that $x$ and $y$ become holes and $z$ has a peg. We denote this jump by $xyz$. 

In general, if we start with some configuration of pegs and holes, and some succession of jumps reduces the number of pegs to 1, then the configuration is \emph{solvable}. In the peg solitaire game on a graph $G$, if some configuration with a hole at one vertex and pegs at all other vertices is solvable, then we say $G$ is \emph{solvable}. If $G$ can be solved starting with a single hole at any vertex, then $G$ is \emph{freely solvable}. Solvability requires $G$ to be connected.\footnote{There are several traditional boards marketed commercially, a triangle with 15 positions in the U.S., a portion of a grid in England (marketed as ``Hi-Q'' in the U.S.), and a European board with more positions than the U.S. board. The significant distinction between these games and the graph version is that they restrict jumps to be made along geometric straight lines.}

Beeler and Hoilman~\cite{BH} determined which graphs are solvable and freely solvable among stars, paths, cycles, complete graphs, and complete bipartite graphs. 
They also proved that the cartesian products of solvable graphs are solvable and gave additional sufficient conditions for the solvability of cartesian products of graphs. Walvoort~\cite{BW} also determined which of the trees of diameter 4 are solvable. 

An alternate goal for the peg solitaire game was proposed in~\cite{BR}. In the \emph{fool's solitaire} game, we instead try to maximize the number of pegs at the end of the process (when there are no remaining available moves). A \emph{terminal state} is an independent set of vertices that gives the final locations of the pegs when the game is played starting with some configuration having a single hole. The \emph{fool's solitaire number} of a graph $G$ is the maximum size of a terminal state and is denoted $F(G)$. A fundamental observation follows from the fact that moves from a configuration are the reverse of moves from the complementary configuration. 

\begin{prop}\label{important}~\cite{BR} A set of vertices $T$ is a terminal state of some solitaire game on $G$ if and only if a starting configuration with holes at vertices of $T$ and pegs at vertices of $V(G) - T$ can be reduced to a single peg.
\end{prop}

Proposition~\ref{important} is used in our proofs of lower bounds on the fool's solitaire number. 
Letting $\alpha(G)$ denote the independence number of $G$, Beeler and Rodriguez~\cite{BR} also proved

\begin{prop}\label{complement}~\cite{BR} Let $G$ be a graph. Trivially $F(G) \le \alpha(G)$. Also, if $\alpha(G)\le|V(G)|-2$ and $V(G)-A$ is independent whenever $A$ is a maximum independent set, then $F(G) \le \alpha(G) - 1$. 
\end{prop}

The proposition holds since if the complement of every maximum independent set is independent and has at least two vertices, then by Proposition~\ref{important} no maximum independent set can be the terminal state of a solitaire game. 

The fool's solitaire numbers for complete graphs, stars, complete bipartite graphs, paths, cycles, and hypercubes were found in~\cite{BR}. The fool's solitaire number of trees with diameter 4 was computed by Walvoort~\cite{BW}. In particular, there is a class of diameter 4 trees for which $\alpha(G) - F(G)$ approaches $\alpha(G)/6$. 
, disproving an earlier conjecture that $\alpha(G) - F(G) \le 1$. It remains open how small $F(G)$ can be in terms of $\alpha(G)$. 

Beeler and Rodriguez~\cite{BR} proved $F(K_{n,m}) = \alpha(K_{n,m}) -1$, and thus Proposition~\ref{complement} is sharp. In Section~\ref{Joins}, we extend their result on complete bipartite graphs by determining the fool's solitaire number of all graphs whose complements are disconnected. 

Beeler and Rodriguez~\cite{BR} also asked for the behavior of the fool's solitaire number under the cartesian product operation. The \emph{cartesian product} of $G$ and $H$, denoted $G\mybox H$, is the graph with vertex set $V(G)\times V(H)$ such two vertices are adjacent if and only if they are equal in one coordinate and adjacent in the other. In Section~\ref{Cartesian}, we show $F(G \mybox K_k) = \alpha(G \mybox K_k)$ for $k \ge 3$ when $G$ is any connected graph. However, this behavior does not hold when $k=2$: if $G$ is a bipartite graph with a Hamiltonian path, then $F(G \mybox K_2) = \alpha(G \mybox K_2) - 1$. This leads us to ask, 

\begin{question}
What is $F(G \mybox K_2)$ when $G$ is not a bipartite graph having a Hamiltonian path? 
\end{question}

Walvoort~\cite{BW} asked for a non-trivial lower bound on $F(G)$. In this direction, we give sufficient conditions for $F(G \mybox H) \ge F(G) F(H)$ in Section~\ref{LowerBound}. This is a partial answer to the question in~\cite{BR} asking for the relationship among $F(G), F(H),$ and $F(G \mybox H)$. In considering the sharpness of our inequality, we ask, 

\begin{question}
By how much can $F(G \mybox H)$ exceed $F(G) F(H)$? 
\end{question}

Computer testing shows that $F(G \mybox H) \ge F(G)F(H)$ does not always hold: if $G$ is the star with 4 vertices and $H$ is the paw or $P_3$, then $F(G \mybox H) = F(G)F(H) - 1$. This leads to the question 

\begin{question}
When does $F(G)F(H)$ exceed $F(G \mybox H)$? 
\end{question}

\section{Joins}\label{Joins}
The \emph{join} of $G$ and $H$, denoted $G \gjoin H$, is formed by adding to the disjoint union of $G$ and $H$ all edges joining $V(G)$ and $V(H)$. Note that every join in connected and these are precisely the graphs whose complements are disconnected. For the complete bipartite graph $K_{n,m}$ with $n \ge m >1$, Beeler and Rodriguez~\cite{BR} showed $F(K_{n,m}) = n-1$. By viewing $K_{n,m}$ as $\overline{K_n} \gjoin \overline{K_m}$, we expand their methods to find the fool's solitaire number of all graph joins, starting with the case of joins with $K_1$.

\begin{lemma}
\label{K_1}
If $G$ is a graph, then $F(G \gjoin K_1) = \alpha(G \gjoin K_1)$.
\end{lemma}

\begin{proof}
Always $F(G \gjoin K_1) \le \alpha(G \gjoin K_1) = \alpha(G)$, so we must show $F(G \gjoin K_1) \ge\alpha(G \gjoin K_1)$. If $G = \overline{K_n}$, then $G \gjoin K_1$ is a star and $F(G \gjoin K_1) = \alpha(G \gjoin K_1)$ because there are no available moves if we place the starting hole at the center of the star. Otherwise, let $S$ be a largest independent set of $G$, and let $z$ be the vertex outside $G$. We wish to show that $S$ is a terminal state; by Proposition~\ref{important} it suffices to solve the game where $S$ gives the locations of the starting holes. Since $S$ is a maximum independent set, there is a hole adjacent every peg in $G$. Start by jumping any peg in $G$ over the peg at $z$, and landing in a hole adjacent to another peg in $G$. We now have two adjacent pegs and we next jump one over the other and land at the hole at $z$. By repeating this two-jump process the number of pegs is reduced to 1. 
\end{proof}

The remaining case is when $G\gjoin H$ is not a complete bipartite graph and has no dominating vertex.

\begin{theorem}
Let $G$ and $H$ be graphs with $|V(G)|, |V(H)| \ge 2$ and $|E(G)| + |E(H)| \ge 1$. Then $F(G \gjoin H) = \alpha(G \gjoin H)$.
\end{theorem}

\begin{proof}
Always $F(G \gjoin H) \le \alpha(G \gjoin H)$, so we must show $F(G \gjoin H) \ge \alpha(G \gjoin H)$. Without loss of generality, $\alpha(G) \ge \alpha(H)$ so $\alpha(G \gjoin H) = \alpha(G)$. By Proposition~\ref{important} consider the complementary problem, where the holes are at some maximum independent set $S$ in $G$.

If $G$ has an edge, then $G$ has a vertex with a peg and a vertex without a peg. Jump this peg over a peg in $H$ and land in $G$. Because all vertices in $H$ start with pegs and $|V(H)| \ge 2$, there remains a peg in $H$, so $H$ now has a peg and a hole. Jump a peg from $H$ over a peg in $G$, landing in the hole in $H$. Continue jumping from $H$ over $G$ to $H$ until there is exactly one peg left in $G$. Because $|V(G)| \ge 2$, there is a hole in $G$. Now we can jump the peg in $G$ over a peg in $H$ and land in $G$ until all pegs in $H$ are removed. This leaves a single peg.

If $G$ has no edge, then $H$ has an edge $e$. Also $S = V(G)$, so every vertex of $G$ is a hole and every vertex of $H$ has a peg. In particular, both endpoints of $e$ have pegs. Use one to jump over the other and into a hole in $G$. Now because $|V(G)| \ge 2$, we may repeatedly jump the peg in $G$ over a peg in $H$ to land in a hole in $G$ until all pegs in $H$ are gone. This leaves a single peg.
\end{proof}

The proof Beeler and Rodriguez~\cite{BR} give for $F(K_{n,m})$ uses Proposition~\ref{complement} for the upper bound and Proposition~\ref{important} for the lower bound. Assuming that $n \ge m \ge 2$, they start with holes at $n-1$ of the vertices in the larger part and jump the single peg in this part over pegs in the other part until the configuration is reduced to a single peg. Our proof above extends this concept to general graphs. Combining our results with theirs, we have $F(G \gjoin H) = \alpha(G \gjoin H)$ unless $G$ and $H$ are both independent and have at least two vertices, in which case $F(G \gjoin H) = \alpha(G \gjoin H) -1$. This gives the fool's solitaire number of all graphs whose complements are disconnected. 

\section{Cartesian Products}\label{Cartesian}

In this section we find $F(G \mybox K_k)$ for $k \ge 3$ when $G$ is a connected graph. A cartesian product is connected if and only if both factors are connected.

We start with three lemmas that aid in finding $F(G \mybox K_k)$. The first two discuss the location of the final peg when solving a complete graph. Note that since $\alpha(K_n)=1$, the fool's solitaire game on the complete graph is the same as the peg solitaire game on the complete graph.

\begin{lemma}
\label{nice5+}
For $k > 4$, the peg/fool's solitaire game on $K_k$ with initial hole at a specified vertex may end with the final peg at any vertex.
\end{lemma}

\begin{proof}
Let $v$ be the vertex required to be occupied at the end of the game. Since it takes $k-2$ jumps to end the game, at least three jumps occur. 

If $v$ starts with the hole, then the first jump ends with a peg at $v$. With the second jump, we can jump a peg over the peg at $v$ to one of the new holes. Now $K_k-v$ has at least one hole and we can play on $K_k - v$ until two pegs remain. Finally, jump one peg over the other to leave the last peg on $v$.

If $v$ starts with a peg, then we can first jump it over another peg and land in the hole. Now proceed on $K_k-v$ as in the previous case.
\end{proof}

\begin{lemma}
\label{nice4}
The peg/fool's solitaire game on $K_4$ may end with the peg in any location except the location of the starting hole.
\end{lemma}

\begin{proof}
Let $u$ be the location of the initial hole and $v$ be the vertex required to be occupied at the end of the game. Two jumps will end the game. Because the first jump must end with a peg at $u$, the second jump must end with no peg on $u$. Hence we cannot have $v = u$. If $v \neq u$, then the first jump can remove the peg at $v$, and the second jump can land there.
\end{proof}

In contrast, in the peg/fool's solitaire game on $K_3$ there is a single jump. Therefore, the final peg must be at the location of the starting hole. Lacking the flexibility guaranteed by Lemma~\ref{nice5+} and~\ref{nice4}, when studying $G \mybox K_3$ we use a property of the game on $P_2 \mybox K_3$.

\begin{lemma}\label{k3lem}
Given at least one peg and at least one hole in each copy of $K_3$ in $P_2 \mybox K_3$ such that the locations of the starting pegs do not form an independent set, a succession of jumps can end with no pegs on one copy of $K_3$ and at least one peg and one hole on the other copy of $K_3$. If there is only one peg at the end, then there are two possible locations for that peg. 
\end{lemma}

\begin{figure} \centering
\begin{subfigure}[b]{.5\linewidth} \centering
\begin{tikzpicture}[scale = .75]
\path (60:1.3cm) coordinate (X1); \path (180:1.3cm) coordinate (X2); \path (300:1.3cm) coordinate (X3);
\foreach \i in {1,2,3} {\path (X\i)++(2.5,0) coordinate (Y\i); \draw (X\i) -- (Y\i);}
\textcolor{white}{\draw[third arrow] (Y1) -- (Y3) node [pos=0.62, right, scale=.75] {$1$}; }
\draw (X1) -- (X2) -- (X3) -- cycle; 
\draw (Y1) -- (Y2) -- (Y3) -- cycle; 
\draw[third arrow] (Y1) -- (Y2) node [pos=0.62, below, scale=.75] {$1$}; \draw[third arrow] (Y2) -- (Y1) node [pos=0.72, below, scale=.75] {$2$};
\fill (X1) circle (2.66pt); \fill (X2) circle (2.66pt); \fill[fill=white,draw=black] (X3) circle (2.66pt); 
\fill (Y1) circle (2.66pt); \fill (Y2)[fill=white, draw = black] circle (2.66pt); \fill (Y3) circle (2.66pt); 
\end{tikzpicture}
\caption{\textcolor{white}{**}}
\label{k3a}
\end{subfigure}%
~\qquad
\begin{subfigure}[b]{.5\linewidth} \centering
\begin{tikzpicture}[scale=.75]
\path (60:1.3cm) coordinate (X1); \path (180:1.3cm) coordinate (X2); \path (300:1.3cm) coordinate (X3);
\foreach \i in {1,2,3} {\path (X\i)++(2.5,0) coordinate (Y\i); \draw (X\i) -- (Y\i);}
\textcolor{white}{\draw[third arrow] (Y1) -- (Y3) node [pos=0.62, right, scale=.75] {$1$}; }
\draw (X1) -- (X2) -- (X3) -- cycle; 
\draw (Y1) -- (Y2) -- (Y3) -- cycle; 
\draw[third arrow] (X2) -- (Y2) node [pos=0.62, below, scale=.75] {$1$}; 
\fill (X1) circle (2.66pt); \fill (X2) circle (2.66pt); \fill[fill=white,draw=black] (X3) circle (2.66pt); 
\fill (Y1) circle (2.66pt); \fill[fill=white,draw = black] (Y2) circle (2.66pt); \fill[fill=white,draw=black] (Y3) circle (2.66pt); 
\end{tikzpicture}
\caption{\textcolor{white}{**}}
\label{k3b}
\end{subfigure} 

\begin{subfigure}[b]{.5\linewidth} \centering
\begin{tikzpicture}[scale=.75]
\path (60:1.3cm) coordinate (X1); \path (180:1.3cm) coordinate (X2); \path (300:1.3cm) coordinate (X3);
\foreach \i in {1,2,3} {\path (X\i)++(2.5,0) coordinate (Y\i); \draw (X\i) -- (Y\i);}
\draw (X1) -- (X2) -- (X3) -- cycle; 
\draw (Y1) -- (Y2) -- (Y3) -- cycle; 
\draw[third arrow] (Y1) -- (Y3) node [pos=0.62, right, scale=.75] {$1$}; 
\fill (X1) circle (2.66pt); \fill[fill=white,draw=black] (X2) circle (2.66pt); \fill[fill=white,draw=black] (X3) circle (2.66pt); 
\fill (Y1) circle (2.66pt); \fill (Y2) circle (2.66pt); \fill[fill=white,draw = black] (Y3) circle (2.66pt); 
\end{tikzpicture}
\caption{\textcolor{white}{**}}
\label{k3c}
\end{subfigure}%
~\qquad
\begin{subfigure}[b]{.5\linewidth} \centering
\begin{tikzpicture}[scale=.75]
\path (60:1.3cm) coordinate (X1); \path (180:1.3cm) coordinate (X2); \path (300:1.3cm) coordinate (X3);
\foreach \i in {1,2,3} {\path (X\i)++(2.5,0) coordinate (Y\i); \draw (X\i) -- (Y\i);}
\draw (X1) -- (X2) -- (X3) -- cycle; 
\draw (Y1) -- (Y2) -- (Y3) -- cycle; 
\draw[third arrow] (Y3) -- (Y1) node [pos=0.62, right, scale=.75] {$1$}; \draw[third arrow] (Y1) -- (Y3) node [pos=0.62, right, scale=.75] {$2$}; 
\fill (X1) circle (2.66pt); \fill[fill=white,draw=black] (X2) circle (2.66pt); \fill[fill=white,draw=black] (X3) circle (2.66pt); 
\fill[fill=white,draw=black] (Y1) circle (2.66pt); \fill (Y2) circle (2.66pt); \fill (Y3) circle (2.66pt); 
\end{tikzpicture}
\caption{\textcolor{white}{**}}
\label{k3d}
\end{subfigure}
\caption{Cases for peg placement in the proof of Lemma~\ref{k3lem}.} 
\label{k3}
\end{figure}
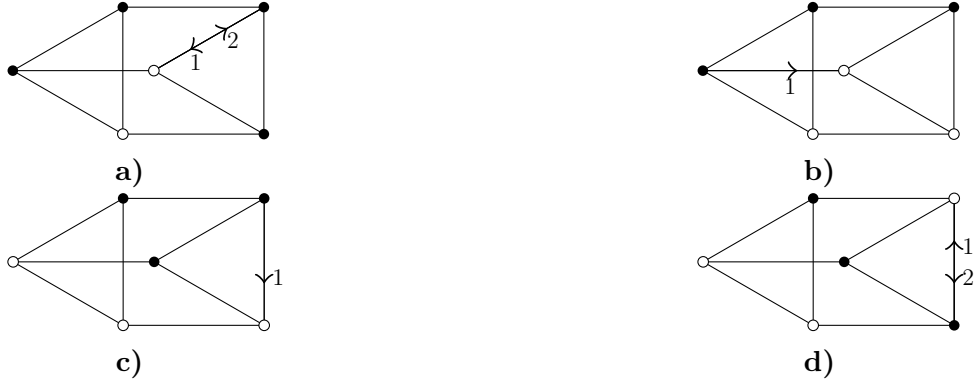

\begin{proof}
Let $T_1$ and $T_2$ be the copies of $K_3$, where $T_1$ is the copy we wish to clear. In each subfigure of Figure~\ref{k3}, $T_1$ is on the left and $T_2$ is on the right. The arrows give the second edge involved in the jump(s) and are numbered to indicate order. 

We first consider the case where there are two pegs on $T_1$. If there are two pegs on $T_2$, then we make two jumps.  The first is within $T_2$ leaving two holes in $T_2$ and allowing us to jump one peg in $T_1$ over the other into $T_2$. An example of this is shown in Figure~\ref{k3a}. If there is one peg in $T_2$, then we start with two holes in $T_2$ and we jump one peg in $T_1$ over the other into $T_2$. An example of this configuration is shown in Figure~\ref{k3b}. Either way, we end with two pegs on $T_2$ and no pegs on $T_1$. 

We next consider the case when there is one peg in $T_1$. First suppose that there are two pegs on $T_2$. By symmetry, either the peg on $T_1$ is adjacent to a peg on $T_2$ or it is not. These cases are illustrated in Figures~\ref{k3c}~and~\ref{k3d}. If possible, we jump the peg on $T_1$ over a peg on $T_2$ into $T_2$. This leaves $T_1$ with no pegs and and $T_2$ with two pegs. If this jump is unavailable, we instead make a jump within $T_2$. This leaves a peg in $T_2$ adjacent to the peg in $T_1$. We can then jump the peg in $T_1$ over the peg in $T_2$ into either hole in $T_2$. Suppose instead that $T_2$ starts with one peg. By our assumption that the locations of the pegs do not form an independent set, the peg on $T_2$ is adjacent to the peg on $T_1$. Jump the peg on $T_1$ over the peg in $T_2$ into either hole in $T_2$. 
\end{proof}

We can now use the lemmas to find the fool's solitaire number of $G \mybox K_k$ when $k \ge 3$. Berge~\cite{B}
proved that $\alpha(G \mybox K_k) = |V(G)|$ if and only if $k \ge X(G)$.

\begin{theorem}
Let $G$ be a connected graph. If $k \ge 3$, then $F(G \mybox K_k) = \alpha(G \mybox K_k)$. In particular, $F(G \mybox K_k) = |V(G)|$ when $k\ge\chi(G)$.
\label{cartesianthm}
\end{theorem}

\begin{proof}
We will denote the copy of $K_k$ that contains all copies of a vertex $v \in V(G)$ by $K(v)$. The vertices of $K(v)$ will be $\{v_1,\ldots,v_k\}$, where $v_i$ plays the role of $v$ in the $i$th copy of $G$. By Proposition~\ref{important}, it suffices to show that some configuration with holes at a maximum independent set can be reduced to a single peg. Start with holes at a maximum independent set $S$ in $G \mybox K_k$. Note that $S$ has at most one vertex in each copy of $K_k$. We preform jumps in two phases. 

Phase 1 achieves a configuration in which each copy of $K_k$ has exactly one hole. Since $S$ is a maximum independent set, for any copy $K(v)$ of $K_k$ having no hole, there is an edge $uv \in E(G)$ such that $K(u)$ has one hole. Jump a peg from $K(v)$ over a peg in $K(u)$ to the hole in $K(u)$. Now both $K(v)$ and $K(u)$ have one hole. The process continues until each copy of $K_k$ has exactly one hole, ending Phase 1.

For Phase 2, let $T$ be a spanning tree of $G$. 

Case 1: $k \ge 4$. Let $v$ be a leaf of $T$, and let $u$ be the neighbor of $v$ in $T$. Let $u_i$ be the vertex of $K(u)$ that has a hole. Since $K(v)$ has a single hole, we may solve $K(v)$. Because $k \ge 4$, by Lemmas~\ref{nice5+} and~\ref{nice4}, we may choose the location of the  final peg on $K(v)$ to be $v_j$ with $i \neq j$. We can then jump $v_ju_ju_i$ (see Figure~\ref{cartesianpic1}). Now $K(u)$ has only the hole at $u_j$. Remove $v$ from $T$ and repeat this process with a new leaf. Continue until the remaining pegs lie in a single complete subgraph, which is solvable. 

Case 2: $k=3$. Because we cannot control the location of the final peg in each copy of $K_3$, the previous strategy does not work, and we instead use Lemma~\ref{k3lem}. When $K(v)$ has one or two pegs, Lemma~\ref{k3lem} allows us to remove all pegs from $K(v)$ and leave one or two pegs on $K(u)$. Remove $v$ from $T$ and repeat with a new leaf. Continue until the remaining pegs lie in a single copy of $K_3$, which is solvable. The only possible problem with this strategy is that Lemma~\ref{k3lem} does not apply when each of $K(v)$ and $K(u)$ has only one peg and they sit at nonadjacent vertices. Since each copy of $K_3$ starts with at least two pegs, this situation arises only for adjacent vertices of $T$ from which neighbors have been eliminated.  

Suppose that $K(v)$ most recently received pegs from $K(x)$ and $K(u)$ most recently received pegs from $K(y)$ (see Figure~\ref{cartesianpic2}. In the applications of Lemma~\ref{k3lem} to the pair $K(x)$ and $K(v)$ and the pair $K(y)$ and $K(u)$, there were two choices for the location of the remaining peg on $K(v)$ and on $K(u)$. Since two element subsets of a set of three indices have a common element, we may choose the moves in the application of Lemma~\ref{k3lem} to $K(x)$ and $K(v)$ and to $K(y)$ and $K(u)$ so that the pegs on $K(v)$ and $K(u)$ are adjacent. Additionally, choosing these moves does not affect future applications of Lemma~\ref{k3lem} involving $K(u)$, since we have two choices for the location of the peg resulting from the application of Lemma~\ref{k3lem} to $K(v)$ and $K(u)$. 
\end{proof}

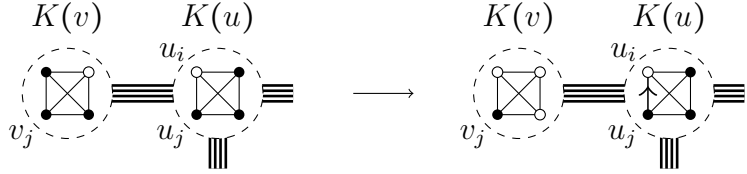
\begin{figure} \centering
\begin{tikzpicture}
\path (0,0) coordinate (X1);
\path (2,0) coordinate (X2); 
\path (3,0) coordinate (X3); 
\path (2,-1) coordinate (X4); 
\draw[line width = 7pt] (X1) -- (X2); \draw[color=white,line width=5pt] (X1) -- (X2); \draw[line width = 3pt] (X1) -- (X2); \draw[color = white, line width=1pt] (X1) -- (X2);
\draw[line width = 7pt] (X2) -- (X3); \draw[color=white,line width=5pt] (X2) -- (X3); \draw[line width = 3pt] (X2) -- (X3); \draw[color = white, line width=1pt] (X2) -- (X3);
\draw[line width = 7pt] (X2) -- (X4); \draw[color=white,line width=5pt] (X2) -- (X4); \draw[line width = 3pt] (X2) -- (X4); \draw[color = white, line width=1pt] (X2) -- (X4);
\foreach \i in {1,2} {\fill[fill=white,draw=black,dashed] (X\i) circle (.6cm);}
\foreach \i in {1,2} {\foreach \j in {1,...,4} {\path (X\i)++(90*\j+45:.4cm) coordinate (Y\j\i);}}
\foreach \i in {1,2} {\draw (Y1\i) -- (Y2\i) -- (Y3\i) -- (Y4\i) -- cycle; 
                          \draw (Y1\i) -- (Y3\i); \draw (Y2\i) -- (Y4\i);
                          \fill (Y1\i) circle (2pt); \fill (Y2\i) circle (2pt);
                          \fill (Y3\i) circle (2pt); \fill (Y4\i) circle (2pt);}
\fill[draw=black,fill=white] (Y12) circle (2pt); 
\fill[draw=black,fill=white] (Y41) circle (2pt); 
\node[above left] at (Y12) {$u_i$}; 
\node[below left] at (Y22) {$u_j$}; 
\node[below left] at (Y21) {$v_j$}; 
\node at (0,1) {$K(v)$}; 
\node at (2,1) {$K(u)$};
\draw[->] (3.8,0) -- (4.6,0); 
\path (6,0) coordinate (U1);
\path (8,0) coordinate (U2); 
\path (9,0) coordinate (U3); 
\path (8,-1) coordinate (U4); 
\draw[line width = 7pt] (U1) -- (U2); \draw[color=white,line width=5pt] (U1) -- (U2); \draw[line width = 3pt] (U1) -- (U2); \draw[color = white, line width=1pt] (U1) -- (U2);
\draw[line width = 7pt] (U2) -- (U3); \draw[color=white,line width=5pt] (U2) -- (U3); \draw[line width = 3pt] (U2) -- (U3); \draw[color = white, line width=1pt] (U2) -- (U3);
\draw[line width = 7pt] (U2) -- (U4); \draw[color=white,line width=5pt] (U2) -- (U4); \draw[line width = 3pt] (U2) -- (U4); \draw[color = white, line width=1pt] (U2) -- (U4);
\foreach \i in {1,2} {\fill[fill=white,draw=black,dashed] (U\i) circle (.6cm);}
\foreach \i in {1,2} {\foreach \j in {1,...,4} {\path (U\i)++(90*\j+45:.4cm) coordinate (V\j\i);}}
\draw[third arrow] (V22) -- (V12); 
\foreach \i in {1,2} {\draw (V1\i) -- (V2\i) -- (V3\i) -- (V4\i) -- cycle; 
                          \draw (V1\i) -- (V3\i); \draw (V2\i) -- (V4\i);
                          \fill (V1\i) circle (2pt); \fill (V2\i) circle (2pt);
                          \fill (V3\i) circle (2pt); \fill (V4\i) circle (2pt);}
\fill[draw=black,fill=white] (V12) circle (2pt); 
\fill[draw=black,fill=white] (V41) circle (2pt); \fill[draw=black,fill=white] (V31) circle (2pt);  \fill[draw=black,fill=white] (V11) circle (2pt); 
\node[above left] at (V12) {$u_i$}; 
\node[below left] at (V22) {$u_j$}; 
\node[below left] at (V21) {$v_j$}; 
\node at (6,1) {$K(v)$}; 
\node at (8,1) {$K(u)$};
\end{tikzpicture}
\caption{Case 1 of Phase 2 in the proof of Theorem~\ref{cartesianthm}.}
\label{cartesianpic1}
\end{figure}

\begin{figure} \centering
\begin{tikzpicture}
\path (0,0) coordinate (X1);
\path (2,0) coordinate (X2); 
\path (4,0) coordinate (X3); 
\path (6,0) coordinate (X4); 
\draw[line width=5pt] (X1) -- (X4); \draw[color=white,line width = 3pt] (X1) -- (X4); \draw[line width=1pt] (X1) -- (X4);
\foreach \i in {1,...,4} {\fill[fill=white,draw=black,dashed] (X\i) circle (.6cm);}
\foreach \i in {1,...,4} {\foreach \j in {1,...,3} {\path (X\i)++(120*\j-30:.4cm) coordinate (Y\j\i);}}
\foreach \i in {1,...,4} {\draw (Y1\i) -- (Y2\i) -- (Y3\i) -- cycle; 
                          \fill[draw=black,fill=white] (Y1\i) circle (2pt); 
													\fill[draw=black,fill=white] (Y2\i) circle (2pt); 
													\fill[draw=black,fill=white] (Y3\i) circle (2pt);}
\fill (Y12) circle (2pt); 
\fill (Y23) circle (2pt); 
\node at (0,1) {$K(x)$}; 
\node at (2,1) {$K(v)$}; 
\node at (4,1) {$K(u)$}; 
\node at (6,1) {$K(y)$}; 
\draw[->] (.4, -.8) -- (1.6,-.8); 
\draw[->] (5.6,-.8) -- (4.4,-.8); 
\draw[->, dashed] (2.4,-.8) -- (3.6,-.8); 
\end{tikzpicture}
\caption{Possible problem in Case 2 of Phase 2 in the proof of Theorem~\ref{cartesianthm}.}
\label{cartesianpic2}
\end{figure}
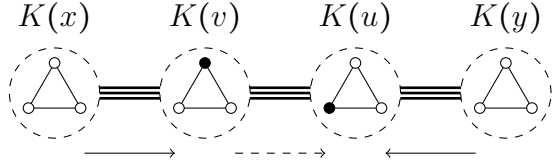

The methods above do not work when $k=2$. As a corollary to Theorem~\ref{k2sad}, we show that $k \ge 3$ is required to guarantee that $F(G \mybox K_k) = \alpha(G \mybox K_k)$ for every graph $G$. 

\begin{lemma}\label{IHaveCors}
If $H$ is a connected, $n$-vertex, bipartite graph having a Hamiltonian path and at least four vertices, then $F(H) \ge \lceil \frac{n}{2} \rceil - 1=\alpha(H)-1$.
\end{lemma}

\begin{proof}
Let $\{v_1,\ldots,v_n\}$ in order form a Hamiltonian path in $H$. Because $H$ is bipartite and has a Hamiltonian path, $\alpha(H)=\left\lceil \frac{n}{2} \right\rceil$. To show that $F(H) \ge \alpha(H) - 1$, we claim that the set of odd-indexed vertices other than $v_1$ forms a terminal state. To show this we solve the game that starts with pegs at the even-indexed vertices and at $v_1$. To solve this configuration, jump the peg at $v_1$ over the pegs at the even-indexed vertices from smallest index to largest index. If $n$ is odd, then the process ends with this peg at $v_n$ and no other pegs. If $n$ is even, then the process ends with this peg at $v_{n-1}$ and a peg at $v_n$. Performing the jump $v_n v_{n-1} v_{n-2}$ leaves a single peg.
\end{proof}

\begin{theorem}
\label{k2sad}
If $G$ is a connected, bipartite graph having a Hamiltonian path and at least two vertices, then $F(G \mybox P_k) = \alpha(G\mybox P_k)-1$ for $k\ge2$. 
\end{theorem}

\begin{proof}
Let $X \cup Y$ be the bipartition of $G$. Without loss of generality, we may assume $|X| = \lceil \frac{n}{2} \rceil$, and $|Y| = \lfloor \frac{n}{2} \rfloor$, where $n = |V(G)|$. In $G \mybox P_k$ we have $k$ copies of $G$, say $G_1,\dots,G_k$, corresponding to the vertices of $P_k$. Let $S$ be the set of vertices of $X$ in $G_i$ for odd $i$ and vertices of $Y$ in $G_i$ for even $i$; $S$ is an independent set of size $\lceil\frac{k}2\rceil|X|+\lfloor\frac{k}2\rfloor|Y|$ in $G \Box P_k$. This forms a maximum independent set because an independent set can contain a copy of $v \in V(G)$ in at most one of $G_i$ and $G_{i+1}$. Note that $V(G \mybox P_k) - S$ is also an independent set. 

Let $S'$ be another independent set of size $\lceil\frac{k}2\rceil \lceil \frac{n}{2} \rceil+\lfloor\frac{k}2\rfloor \lfloor \frac{n}{2} \rfloor$. 

If $k$ is odd, then the only way for $\lceil \frac{k}{2} \rceil$ copies of a vertex of $G$ to appear in $S'$ is for them to appear in every odd-indexed copy of $G$. Thus, there are $\lceil \frac{n}{2} \rceil$ vertices of $G$ whose odd-indexed copies all appear in $S'$. These vertices cannot be consecutive on the Hamiltonian path $P$ in $G$. Furthermore, if more than two consecutive vertices along $P$ have only $\lfloor \frac{k}{2} \rfloor$ copies in $S'$, then $S'$ is too small. Furthermore, suppose that two consecutive vertices $x$ and $y$ along $P$ both have only $\lfloor \frac{k}{2} \rfloor$ copies in $S'$. If $x$ and $y$ are not the first or last two vertices of $P$, then their other neighbors along $P$ must both have $\lceil \frac{k}{2} \rceil$ copies in $S'$. This places the neighbors in the odd-indexed copies of $G$ which permits only even-indexed copies of $x$ and $y$ to appear in $S'$. However, since $S'$ is an independent set and $x$ and $y$ are adjacent in $G$, $S'$ cannot contain the copies of $x$ and $y$ in any single copy of $G$, so this restriction to the even-indexed copies of $x$ and $y$ means $S'$ cannot have the desired size. Suppose instead that $x$ is an endpoint of $P$. Then the neighbor of $y$ along $P$ other than $x$ must have $\lceil \frac{k}{2} \rceil$ copies in $S'$. This places the neighbor in the odd-indexed copies of $G$ which forces $y$ to appear only in even-indexed copies of $G$. Since $S'$ has $\lfloor \frac{k}{2} \rfloor$ copies of $y$, every copy of $x$ in $S'$ must be in an odd-indexed copy of $G$. Taking all odd-indexed copies of $x$ gives a larger independent set, contradicting the choice of $S'$. We conclude that every vertex in a largest partite set of $G$ appears $\lceil \frac{k}{2} \rceil$ times in $S'$, so that $S'$ is $S$ of $V(G \mybox P_k)-S$. 

If $k$ is even, then there $S'$ contains $\frac{k}{2}$ copies of every vertex. Hence, $S'$ contains either the odd-indexed copies or the even-indexed copies of a vertex. Along the Hamiltonian path, these must alternate, thus $S'$ is either $S$ or $V(G\mybox P_k)-S$.  

 As $G$ has at least two vertices and $k\ge2$, each of $S$ and $V(G \mybox P_k) - S$ has at least two vertices. Thus by Proposition~\ref{complement}, $F(G \mybox P_k) \le \alpha(G \mybox P_k) -1$. Furthermore, $G \mybox P_k$ is a connected, bipartite graph with a Hamiltonian path and at least four vertices, so by Theorem~\ref{IHaveCors}, $F(G \mybox P_k) = \alpha(G \mybox P_k) - 1$. 
\end{proof}

\begin{cor}
$F(G \mybox K_2) = \alpha(G \mybox K_2) -1$ if $G$ is a connected, bipartite graph having a Hamiltonian path and at least two vertices.
\end{cor}

\section{A Product Lower Bound}\label{LowerBound}

Beeler and Rodriguez \cite{BR} asked what can be said about the value of $F(G \mybox H)$ in terms of $F(G)$ and $F(H)$. We obtain a sufficient condition for $F(G \mybox H) \ge F(G) F(H)$. Let $N[v]$ denote the closed neighborhood $N(v)\cup\{v\}$ of a vertex $v$. We say a graph $G$ is \emph{freely neighborhood-solvable} if, for every $v \in V(G)$, $G$ is solvable from the position with a single hole at $v$ so that the final peg is in $N[v]$. Graphs previously known to be freely solvable that have this stronger property include complete graphs, even cycles of length up to 10 ($C_{12}$ is not), the platonic solids, and the Petersen graph. Beeler and Gray~\cite{BG} found that 103 of the 112 six-vertex graphs and 820 of the 853 seven-vertex graphs are freely solvable. Computer search shows that 95 of these 103 and 796 of these 820 are freely neighborhood-solvable. Additionally, over 98\% of eight-vertex and nine-vertex graphs are freely neighborhood-solvable. 

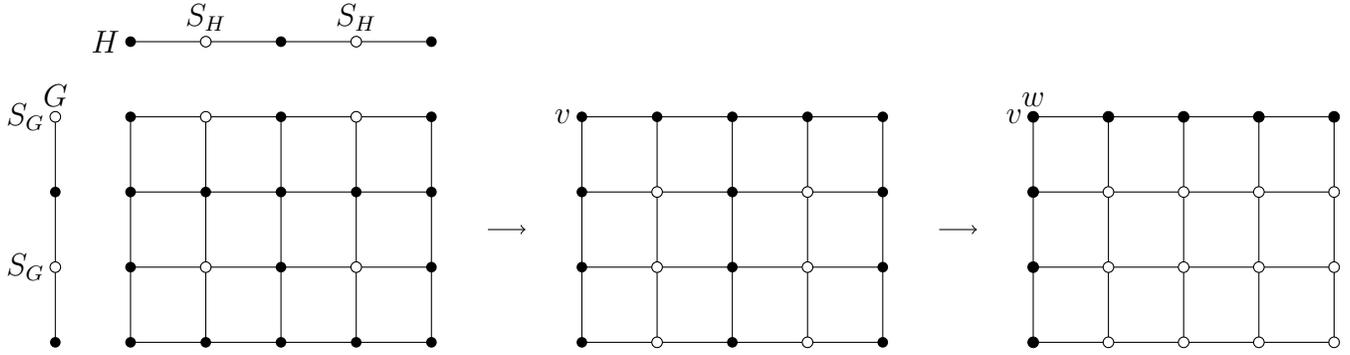
\begin{figure} 
\makebox[\textwidth][c]{
\begin{tikzpicture}
\foreach \i in {0,...,5} {\draw (\i,0) -- (\i,1) -- (\i,2) -- (\i,3);}
\foreach \j in {0,...,4} {\draw (1,\j) -- (2,\j) -- (3,\j) -- (4,\j) -- (5,\j);}
\foreach \i in {0,...,5} {\foreach \j in {0,...,4} {\fill (\i,\j) circle (2pt);}}
\foreach \i in {0,2,4} {\foreach \j in {1,3,4} {\fill[fill=white,draw=black] (\i,\j) circle (2pt);}}
\node[left] at (1,4) {$H$}; 
\node[above] at (2,4) {$S_H$};
\node[above] at (4,4) {$S_H$};
\node[above] at (0,3) {$G$};
\node[left] at (0,3) {$S_G$};
\node[left] at (0,1) {$S_G$}; 
\fill[fill=white] (0,4) circle (3pt); 
\draw[->] (5.75,1.5) -- (6.25,1.5); 
\foreach \i in {7,...,11} {\draw (\i,0) -- (\i,1) -- (\i,2) -- (\i,3);}
\foreach \j in {0,...,3} {\draw (7,\j) -- (8,\j) -- (9,\j) -- (10,\j) -- (11,\j);}
\foreach \i in {7,...,11} {\foreach \j in {0,...,3} {\fill (\i,\j) circle (2pt);}}
\foreach \i in {8,10} {\foreach \j in {0,1,2} {\fill[fill=white,draw=black] (\i,\j) circle (2pt);}}
\node[left] at (7,3) {$v$};  
\draw[->] (11.75,1.5) -- (12.25,1.5); 
\foreach \i in {13,...,17} {\draw (\i,0) -- (\i,1) -- (\i,2) -- (\i,3);}
\foreach \j in {0,...,3} {\draw (13,\j) -- (14,\j) -- (15,\j) -- (16,\j) -- (17,\j);}
\foreach \i in {13,...,17} {\foreach \j in {0,...,3} {\fill[fill=white,draw=black] (\i,\j) circle (2pt);}}
\foreach \i in {13,...,17}  {\fill (\i,3) circle (2pt);}
\foreach \j in {0,...,3}  {\fill (13,\j) circle (2pt);}
\node[above] at (13,3) {$w$}; 
\node[left] at (13,3) {$v$};
\end{tikzpicture} } 
\caption{The first steps of the proof of Theorem~\ref{lower}.}
\label{GBoxH}
\end{figure}

\begin{figure} \centering
\begin{tikzpicture}
\foreach \i in {1,...,5} {\draw (\i,0) -- (\i,1) -- (\i,2) -- (\i,3);}
\foreach \j in {0,...,3} {\draw (1,\j) -- (2,\j) -- (3,\j) -- (4,\j) -- (5,\j);}
\foreach \i in {1,...,5} {\foreach \j in {0,...,3} {\fill[fill=white,draw=black] (\i,\j) circle (2pt);}}
\foreach \i in {1,...,5}  {\fill (\i,3) circle (2pt);}
\foreach \j in {0,...,3}  {\fill (1,\j) circle (2pt);}
\node[above] at (1,3) {$w$}; \node[above] at (2,3) {$w'$};
\node[left] at (1,3) {$v$}; \node[left] at (1,2) {$v'$}; 
\fill[fill=white,draw=black] (1,3) circle (2pt);  \fill (1,2) circle (2pt); 
\fill[fill=white,draw=black] (2,3) circle (2pt);  \fill (2,2) circle (2pt); 
\draw[->] (5.75,1.5) -- (6.25,1.5); 
\foreach \i in {7,...,11} {\draw (\i,0) -- (\i,1) -- (\i,2) -- (\i,3);}
\foreach \j in {0,...,3} {\draw (7,\j) -- (8,\j) -- (9,\j) -- (10,\j) -- (11,\j);}
\foreach \i in {7,...,11} {\foreach \j in {0,...,3} {\fill[fill=white,draw=black] (\i,\j) circle (2pt);}}
\foreach \i in {7,...,11}  {\fill (\i,3) circle (2pt);}
\foreach \j in {0,...,3}  {\fill (7,\j) circle (2pt);}
\node[above] at (7,3) {$w$}; \node[above] at (8,3) {$w'$};
\node[left] at (7,3) {$v$}; \node[left] at (7,2) {$v'$}; 
\fill[fill=white,draw=black] (7,3) circle (2pt);  \fill[fill=white,draw=black]  (7,2) circle (2pt); 
\fill (8,3) circle (2pt);  \fill[fill=white,draw=black]  (8,2) circle (2pt); 
\draw[->] (11.75,1.5) -- (12.25,1.5); 
\foreach \i in {13,...,17} {\draw (\i,0) -- (\i,1) -- (\i,2) -- (\i,3);}
\foreach \j in {0,...,3} {\draw (13,\j) -- (14,\j) -- (15,\j) -- (16,\j) -- (17,\j);}
\foreach \i in {13,...,17} {\foreach \j in {0,...,3} {\fill[fill=white,draw=black] (\i,\j) circle (2pt);}}
\foreach \j in {0,1,3}  {\fill (13,\j) circle (2pt);}
\node[above] at (13,3) {$w$}; \node[above] at (14,3) {$w'$};
\node[left] at (13,3) {$v$}; \node[left] at (13,2) {$v'$};  
\end{tikzpicture}
\caption{Completion of Theorem~\ref{lower} when solving $H(v)$ ends at $(v,w)$.}
\label{GBoxH1}
\end{figure}

\begin{figure} \centering
\begin{tikzpicture}
\foreach \i in {1,...,5} {\draw (\i,0) -- (\i,1) -- (\i,2) -- (\i,3);}
\foreach \j in {0,...,3} {\draw (1,\j) -- (2,\j) -- (3,\j) -- (4,\j) -- (5,\j);}
\foreach \i in {1,...,5} {\foreach \j in {0,...,3} {\fill[fill=white,draw=black] (\i,\j) circle (2pt);}}
\foreach \i in {1,...,5}  {\fill (\i,3) circle (2pt);}
\foreach \j in {0,...,3}  {\fill (1,\j) circle (2pt);}
\node[above] at (1,3) {$w$}; \node[above] at (2,3) {$w'$};
\node[left] at (1,3) {$v$}; \node[left] at (1,2) {$v'$}; 
\fill[fill=white,draw=black] (1,3) circle (2pt);  \fill[fill=white,draw=black] (1,2) circle (2pt); 
\fill (2,3) circle (2pt);  \fill (2,2) circle (2pt); 
\draw[->] (5.75,1.5) -- (6.25,1.5); 
\foreach \i in {7,...,11} {\draw (\i,0) -- (\i,1) -- (\i,2) -- (\i,3);}
\foreach \j in {0,...,3} {\draw (7,\j) -- (8,\j) -- (9,\j) -- (10,\j) -- (11,\j);}
\foreach \i in {7,...,11} {\foreach \j in {0,...,3} {\fill[fill=white,draw=black] (\i,\j) circle (2pt);}}
\foreach \j in {0,1}  {\fill (7,\j) circle (2pt);}
\node[above] at (7,3) {$w$}; \node[above] at (8,3) {$w'$};
\node[left] at (7,3) {$v$}; \node[left] at (7,2) {$v'$}; 
\fill(8,3) circle (2pt);  \fill  (8,2) circle (2pt); 
\draw[->] (11.75,1.5) -- (12.25,1.5); 
\foreach \i in {13,...,17} {\draw (\i,0) -- (\i,1) -- (\i,2) -- (\i,3);}
\foreach \j in {0,...,3} {\draw (13,\j) -- (14,\j) -- (15,\j) -- (16,\j) -- (17,\j);}
\foreach \i in {13,...,17} {\foreach \j in {0,...,3} {\fill[fill=white,draw=black] (\i,\j) circle (2pt);}}
\foreach \j in {0,1,3}  {\fill (13,\j) circle (2pt);}
\node[above] at (13,3) {$w$}; \node[above] at (14,3) {$w'$};
\node[left] at (13,3) {$v$}; \node[left] at (13,2) {$v'$};  
\end{tikzpicture}
\caption{Completion of Theorem~\ref{lower} when solving $H(v)$ ends at $(v,w')$.}
\label{GBoxH2}
\end{figure}

\begin{theorem}\label{lower}
If $G$ is a freely solvable graph, and $H$ is a freely neighborhood-solvable graph, then $F(G \mybox H) \ge F(G)F(H)$.
\end{theorem}

\begin{proof}
For $v \in V(G)$, let $H(v)$ be the copy of $H$ associated with $v$. Similarly, for $w \in V(H)$, let $G(w)$ be the copy of $G$ associated with $w$. Let $S_G$ be a maximum-sized terminal state for $G$ and let $S_H$ be a maximum-sized terminal state for $H$. By Proposition~\ref{important}, to show that $S_G \times S_H$ is a terminal state in $G \mybox H$, it suffices to solve the configuration with holes at $S_G \times S_H$. Figure~\ref{GBoxH} shows the first steps of the proof. Note that the figures are for illustrative purposes and are not meant to be an example, specifically paths do not meet the given criteria. 

For each $x \in S_H$, we know that $G(x)$ is solvable from this configuration of holes, by Proposition~\ref{important}. Solve these so that each such $G(x)$ leaves its final peg in the same location, say its copy of $v$. Now all vertices of $H(v)$ have pegs and every copy of $H$ except $H(v)$ has holes at the vertices of $S_H$. Now by Proposition~\ref{important} we may solve every copy of $H$ except $H(v)$ so that the final pegs all end up at a copy of the same vertex in $H$; call it $w$. This initial portion of the procedure does not use the properties assumed for $G$ and $H$. 

At this point, there are pegs on all vertices of $H(v)$ and $G(w)$ and no pegs on any other vertices. By our assumption, $H$ can be solved starting with a hole at $w$ and ending with a peg in $N[w]$. 

First, suppose that the final peg in this solution ends at $w$. Let $w'$ be any neighbor of $w$ in $H$ and let $v'$ be any neighbor of $v$ in $G$. Now $\{(v,w), (v',w), (v',w'), (v,w')\}$ induce a 4-cycle in $G \mybox H$. Make the jumps in Figure~\ref{GBoxH1}: jump $(v,w) (v,w') (v',w')$ and then $(v',w) (v',w') (v,w')$. Now $H(v)$ has a hole at $(v,w)$ and at no other location. Solve $H(v)$ so that the final peg ends at $(v,w)$. Now the remaining pegs occur on $V(G(w))-\{(v',w)\}$. Because $G$ is freely solvable, we may solve $G(w)$ and thus solve $G \mybox H$.

Suppose instead that when we solve $H$ from a hole at $w$ that we end with a peg at $w'$, a neighbor of $w$ in $H$. Let $v'$ be any neighbor of $v$ in $G$. Again $\{(v,w), (v',w), (v',w'), (v,w')\}$ induces a 4-cycle in $G \mybox H$. Starting from the state with pegs only on $H(v)$ and $G(w)$ (see last stage of Figure~\ref{GBoxH}), make the jumps in Figure~\ref{GBoxH2}: start with jump $(v,w) (v',w) (v',w')$. Now $H(v)$ has a hole at $(v,w)$ and at no other location. Solve $H(v)$ so that the final peg ends at $(v,w')$. Then jump $(v',w')(v,w')(v,w)$. Now the remaining pegs are on $V(G(w))-(v',w)$. Because $G$ is freely solvable, we may solve $G(w)$ and thus solve $G \mybox H$. 
\end{proof}

In the theorem above, $G$ only needs to satisfy that the configuration having a hole only at some neighbor of $v$ is solvable, where $v$ is a vertex to which the configuration with holes at $S_G$ is solvable.

Computer testing shows that the inequality in Theorem~\ref{lower} is sharp for $P_2 \mybox P_2$ and $(K_4-e) \mybox (K_4-e)$ but not for $(K_4-e) \mybox C_4$, $C_4 \mybox C_4$ or $P_2 \mybox C_n$ when $n \in \{4,6,8\}$. However, this inequality does not hold for every graph: the fool's solitaire number of the cartesian product of $K_{1,3}$ with either $P_3$ or the paw $G$ is less than $F(K_{1,3})F(P_3)$ or $F(K_{1,3})F(G)$ respectively.

\section{Acknowledgments}

Both authors took part in the Combinatorics REGS group at the University of Illinois at Urbana--Champaign during the summer of 2013 and acknowledge support from National Science Foundation grant DMS 08-38434 ``EMSW21-MCTP: Research Experience for Graduate Students.'' We thank Thomas Mahoney and Gregory J. Puleo for examples that motivated our work on graph joins. We also thank Douglas B. West for listening to our ideas and providing a great deal of advice on the editing of this paper.

\bibliographystyle{plain}

\bibliography{FoolsSolitaire}

\end{document}